\documentclass[oneside,a4paper,12pt]{article}
\usepackage[utf8]{inputenc} 

\usepackage{lmodern}
\usepackage[T1]{fontenc}
\usepackage[svgnames,x11names,dvipsnames]{xcolor} 
\usepackage{amsfonts,amsmath,latexsym,amssymb,amsthm,epigraph, bbm, mathtools, upgreek}  
\usepackage{authblk}
\usepackage{xparse, stmaryrd}
\usepackage{tikz-cd, tikz} 
\usepackage{enumerate} 
\usepackage[]{hyperref} 
\usepackage{cleveref} 
\usepackage{indentfirst} 
\usepackage[text={16cm,23cm}]{geometry} 
\usepackage{breqn}
\usepackage{qtree}
\usepackage{listings}

\newcommand{\N}{{\mathbb{N}}} 
\newcommand{\Z}{{\mathbb{Z}}} 
\newcommand{\F}{{\mathbb{F}}}

\NewDocumentCommand{\var}{g}{
	\ensuremath{\operatorname{Var}_{\IfNoValueTF{#1}{k}{#1}}}%
}

\theoremstyle{plain} 
\newtheorem{theorem}{Theorem}[section]

\newtheorem{lemma}[theorem]{Lemma}
\newtheorem{corollary}[theorem]{Corollary}

\newtheorem*{main*}{Main Theorem}

\theoremstyle{definition} 
\newtheorem{definition}[theorem]{Definition}
\newtheorem{remark}[theorem]{Remark}

\newcommand{\Aut}{\mathrm{Aut}}

\title{On the existence of a morphism between certain 
Artin-Schreier curves}
\author{Beatriz Barbero Lucas\thanks{School of Mathematics and Statistics, University College Dublin, Ireland}\quad Stefano Lia \thanks{Department of Mathematics and Mathematical Statistics, Ume\aa~University, Sweden}\quad Gary McGuire\thanks{School of Mathematics and Statistics, University College Dublin, Ireland}}
\date{}
\hypersetup{ 
	pdftitle={On the existence of a morphism between certain 
Artin-Schreier curves},
	pdfauthor={Beatriz Barbero Lucas, Stefano Lia, Gary McGuire},
	bookmarksnumbered=true,     
	bookmarksopen=true,         
	bookmarksopenlevel=1,       
	colorlinks=true,
	linkcolor=blue,
	citecolor=Green,
	urlcolor=blue,
	pdfstartview=Fit,           
	pdfpagemode=UseOutlines,    
}

\begin{document}

\baselineskip18pt
\parskip15pt
\parindent0pt

\maketitle

\begin{abstract}
    It is well known that, given  two  
    curves  $\mathcal{X}: y^p+cy=x^m$ and $\mathcal{Y}:y^p+cy=x^n$, 
    defined over $\F_p$, if $n$ divides $m$ then there exists a nonconstant morphism 
    $\mathcal{X} \longrightarrow \mathcal{Y}$. In this paper we are interested in studying whether the converse of this statement is true, i.e., if there exists a morphism $\mathcal{X} \longrightarrow \mathcal{Y}$ then must it be true that $n$ divides $m$? In particular, we consider the case when $m=p^{k}+1$ and $n=p^\ell+1$. We prove that the converse is true 
under certain hypotheses. We deal with both the cases of Galois morphisms
and non-Galois morphisms.

\end{abstract}

\section{Introduction}

Let $K$ be the algebraic closure of the field $\mathbb{F}_q$ of $q$ elements, where $q$ is a power of a prime $p$.
In this paper we will assume that $p$ is odd.
Let $\mathcal{X}, \mathcal{Y}$ be non-singular, projective, geometrically irreducible algebraic curves of genus $\mathfrak{g}(\mathcal{X}),\mathfrak{g}(\mathcal{Y})\geq 1$ and defined over $K$.
The existence  of non-constant  morphisms 
$\phi~\colon \mathcal{X}\longrightarrow \mathcal{Y}$ is a notoriously hard problem, and especially in the case of non existence, few results are known.
We  will assume that morphisms are defined over $K$, and so any 
non-constant morphism will be surjective.
We will assume that all morphisms are non-constant.

As an instance, consider the family $\mathcal{C}$ 
of plane curves $C_m$ given by separate polynomials
with affine equation of the form
\[
C_m~\colon~y^p+c y=x^m
\]
where $c \in K$ is nonzero, $m>p$, and $\gcd(m,p)=1$. 
When $c=-1$, these are Artin-Schreier curves.
These curves have one singular point, so when
referring to $C_m$ what we  mean is the nonsingular model of the curve,
which has the same function field $\mathrm{F}_m=K(x,y)$, where $y^p+c y=x^m$.
It is straightforward to see that, if $n\mid m$, then there exist a morphism 
$C_m\longrightarrow C_n$, corresponding to the function field 
extension $K(x,y)~:K(x^{m/n},y)$.

It is natural to ask the converse question, namely, whether it is true that, if $n\nmid m$, then no morphism $C_m\longrightarrow C_n$ exists.
Even though the question is very natural, and the family $\mathcal{C}$ has been widely studied \cite{WeierstrassPointsGarciaViana,Bouw2016,van1990reed,HKT}, to the best of our knowledge, there is no result in this direction.

In this paper, we initiate the investigation of this problem, and we address the question for a subfamily of the family $\mathcal{C}$. In particular, we consider the family $\mathcal{B}$ of curves $B_k:=B_{k,c}$, for $k\in \mathbb{N}_{\geq 1}$, given by the normalization of the curve with affine equation
\[
y^p+c y=x^{p^k+1}.
\]
The function field of $B_k$ is $K(x,y)$, which has genus $p^k(p-1)/2$.
We will pay particular attention to the cases $c=+1$ and $c=-1$ (when the curves
are Artin-Schreier curves).

We conjecture that $p^\ell +1$ dividing $p^k+1$ is both a necessary and sufficient condition 
for a morphism $ B_{k} \longrightarrow B_{\ell}$ to exist.
In this article, we present some results to support this conjecture. 
Our main result in this direction is the following, for a morphism
that is not necessarily Galois.

\begin{theorem}\label{intro1}
Let $p$ be an odd prime. Let $\ell, k$ be positive integers with $1<\ell <k$. 
Let $K=\overline{\F_p}$. Let $c \in K$ be nonzero.
Let $B_k$ be the curve $y^p+c y=x^{p^k+1}$ and let 
$B_\ell$ be the curve $y^p+c y=x^{p^\ell +1}$. 
Let $P_k$ and $P_\ell$ be the points at infinity of $B_k$ and $B_\ell$,
respectively.

(a) If $p^\ell +1$  divides $p^k+1$, then
     there exists a non-trivial morphism  
    $\rho : B_{k} \longrightarrow B_{\ell}$, with the properties that
    $\rho(P_k)=P_\ell$, $\rho$ is totally ramified at $P_k$,
    and $\deg (\rho)=\frac{p^k+1}{p^\ell +1}$.

(b) Conversely, if there exists a non-trivial morphism  
    $\phi : B_{k} \longrightarrow B_{\ell}$, such that $\phi (P_k)=P_\ell$ 
    and $\phi$ is totally ramified at $P_k$, 
    and $\deg (\phi)$ is relatively prime to $p$,
    then $p^\ell +1$  divides $p^k+1$ and
    $\deg (\phi)=\frac{p^k+1}{p^\ell +1}$.
\end{theorem}

A number of results give conditions on the existence of a morphism $\phi$ between curves $\mathcal{X}$ and $\mathcal{Y}$. Firstly, a general upper bound on the degree $d$ of $\phi$ can be derived from the Riemann-Hurwitz formula \cite{hartshorne1977algebraic}[Cor. 2.4], namely $d\leq \frac{\mathfrak{g}(\mathcal{X})-1}{\mathfrak{g}(\mathcal{Y})-1}$,
by using the fact that the different is  nonnegative.
Secondly,  by a theorem of Kleiman-Serre \cite{kleiman1968algebraic}[Prop. 1.2.4], if a morphism 
$\phi~\colon \mathcal{X}\longrightarrow \mathcal{Y}$ exists, then the L-polynomial $L_{\mathcal{Y}}(t)$ of $\mathcal{Y}$ must divide the L-polynomial $L_{\mathcal{X}}(t)$ of $\mathcal{X}$. The converse of this result is known to be false.
Interestingly, for the family $\mathcal{B}$ when $c=-1$, 
the Artin-Schreier case, it has been proved in \cite{MCGUIRE20193341}[Theorems 24 and 25] that the L-polynomial of $B_\ell$ divides the L-polynomial of $B_k$ if and only if $\ell$ divides $k$. 
It follows that if $k$ does not divide $\ell$ then we know that there is no morphism
$B_k \longrightarrow B_\ell$.
The conditions in Theorem \ref{intro1} point towards a different (and stronger) condition 
for the existence of a morphism $ B_{k} \longrightarrow B_{\ell}$, namely,
that $p^\ell +1$  divides $p^k+1$.
Therefore, our results provide an infinite family of explicit examples where the converse of the
Kleiman-Serre theorem fails, and also where it succeeds.

In the second part of the paper, we will restrict ourselves to the cases when $c=\pm 1$ and the morphism $\phi: B_k \longrightarrow B_\ell$ is Galois. 
We prove some results like the following (this statement is a collection of some results proved separately later).

\begin{theorem}
If $k=2^a$, there is no Galois morphism from a curve of type $B_k$ to one of type $B_{\ell}$. 

\end{theorem}

\begin{theorem}
There is no Galois morphism
$B_{2\ell}\longrightarrow B_\ell$ for any integer $\ell > 2$.
\end{theorem}

The structure of this article will be as follows. 
In \Cref{backg} we introduce notation and some preliminary lemmas.
We recall some results related to  Riemann-Roch spaces and a basis for the Riemann-Roch spaces of certain divisors of the curve $B_k$.

In \Cref{maint} we prove three preliminary lemmas related to bounds for the degree of a map between the curves $B_k$ and $B_\ell$ and we prove the main result for non-Galois morphisms,  \Cref{thm: main}.

In \Cref{sec: Galois morphisms}, we start our study of the case when the morphism between the curves $B_k$ is Galois, this is, when $B_\ell \simeq B_k/G$ for $G$ a subgroup of the automorphism group $\Aut(B_k)$ of $B_k$.
We use the precise description of $\Aut(B_k)$ given in theorems 12.1 and 12.6 of \cite{HKT}  to classify certain $G$ for which we have such a morphism. 
This is  \Cref{first_groups}.
The proof of \Cref{first_groups} (part 1) uses \Cref{thm: main}.

In \Cref{Section5}, we compute the genus of the quotient $B_k/G$ for any subgroup $G\leq \Aut(B_k)$, and we use this result to prove in \Cref{sec: Cond quot} some conditions under which $B_\ell \simeq B_k/G$. 

Finally, in \Cref{sec:inf family} we provide three results: an infinite subfamily of curves $B_k$ for which there is no Galois morphism between any two members; a second result that assure us that there is no Galois map between curves $B_k$ and $B_\ell$ when $k< \ell^2$ except in the known case $|G|=(p^k+1)/(p^\ell+1)$; and the non-existence of Galois cover between the curve $B_{2\ell}$ and $B_\ell$ when $\ell>2$.

\section{Background}\label{backg}

\subsection{Relevant notions}
As above, we denote by $\mathcal{X}$ a non-singular, projective, geometrically irreducible algebraic curve of genus $\mathfrak{g}(\mathcal{X})\geq 1$ and defined over $K$.
We denote by $K(\mathcal{X})$ the field of rational functions over $\mathcal{X}$.

For every place $P$ of $\mathcal{X}$ we denote its discrete valuation by $v_P$, namely the function $v_P: K(C)^* \longrightarrow \Z$, defined as $v_P(f)=m$ if $f= u \cdot t^m$ for $u$ an invertible element and $t$ a local parameter (uniformizing element) at $P$. If $m>0$, we say that $f$ has a zero of order $m$ at $P$.  If $m<0$, we say that $f$ has a pole of order $-m$ at $P$.

A (Weil) divisor on the curve $\mathcal{X}$ is an element of the free abelian group generated by the set of points of $\mathcal{X}$. This means that a divisor is a formal sum of the form $D=\sum n_i P_i$ with $n_i \in \Z$ and $n_i=0$ for all but a finite number of $P_i$ in $\mathcal{X}$. 
Its degree is defined as $\deg D =\sum n_i$. We say that a divisor $D$ is effective when all $n_i \geq 0$.

\begin{definition}
	If $f$ is a rational function on $\mathcal{X}$, not identically $0$, we define the \textit{divisor of} $f$ to be 
	\[
	\mathrm{div}(f) = \sum_{P \in \mathcal{X}} v_P(f) P.
	\]
    A divisor $D$ is called a \textit{principal divisor} if it is of the form $D=\mathrm{div}(f)$ for any $f\in K(\mathcal{X})^*$.
\end{definition}

Given a morphism $\phi~: \mathcal{X} \longrightarrow \mathcal{Y}$, we denote by $\phi^*$ the pullback of $\phi$. This is an injective homomorphism between the function fields  $\phi^*:K(\mathcal{Y}) \longrightarrow K(\mathcal{X})$
defined by $\phi^*(f)=f \circ \phi$. The pullback map allows one to compute valuations of rational functions on points of $\mathcal{Y}$ in terms of valuations of functions on points of $\mathcal{X}$ and the ramification structure of $\phi$. 

Let $P$ be a place of $\mathcal{X}$. The ramification index of $\phi$ at $P$, denoted by $e_\phi (P)$, is the quantity
\[
e_\phi(P)= v_P(\phi^* (t_{\phi(P)})),
\]
where $t_{\phi(P)} \in K(\mathcal{Y})$ is a uniformizer at $\phi(P)$. Note that $e_\phi(P) \geq 1$. We say that $\phi$
is unramified at $P$ if $e_{\phi}(P) = 1$  and that $\phi$ is unramified if it is unramified at every point of $\mathcal{X}$.
We say that $\phi$ is ramified at $P$ if $e_{\phi}(P) > 1$ and it is totally ramified at $P$ if $e_{\phi}(P) = \deg(\phi)$.

For any rational function $f\in K(\mathcal{Y})$, we have that $v_P(\phi^*(f))=e_{\phi}(P) \cdot v_{\phi(P)}(f)$.

With this notation, the Riemann-Roch space is a $K$-vector space associated to the divisor $D$ on the curve $\mathcal{X}$, and is defined by
$$\mathcal{L}(D)\coloneqq \{f \in K(\mathcal{X})^* \mid \mathrm{div}(f) + D \geq 0\}.$$

\begin{definition}
    For any place $P$ of a curve $\mathcal{X}$, a non-negative integer $n$ is a pole number of $P$ if there is a function $z$ such that $\mathrm{div}(z)_{\infty}= nP$; otherwise, $n$ is a gap number of $P$.  
\end{definition}

\subsection{The curves $B_k$}

Let $\mathcal{B}$ be the family of curves $B_k$, $k\in \mathbb{N}_{\geq 1}$, given by the normalization of the curve with affine equation
\[
y^p+c y=x^{p^k+1}.
\]
In the plane model, the (unique) point at infinity is the only singular point $\tilde{P}_k$, and it has only one branch centered at it. As a consequence, in the nonsingular model $B_k$, we call the (unique) point at infinity of $B_k$ the point $P_k$ corresponding to $\tilde{P}_k$ in $B_k$.
The function field of $B_k$ is $K(B_k)=K(x,y)$, and has genus $\mathfrak{g}_k=p^k(p-1)/2$.
We use the words point and place interchangeably, and we are justified in doing so because we consider nonsingular curves.

The Riemann-Roch spaces supported at $P_k$ are described in the following result, which we will use in the next section.

\begin{theorem}\cite[Prop. 6.4.1]{StichBook}\label{th: basis RR Sp}
    Let $K=\overline{\F_p}$. Let $c \in K$ be nonzero. 
    Let $B_k$ be the curve $y^p+c y=x^{p^k+1}$, and let $P_k$ be the unique point at infinity of $B_k$. 
    The products 
    \[
        x^iy^j \quad \text{ for } i,j \geq 0, \ j\leq p-1, \text{ and } ip+jm \leq n,
    \]
    form a $K$-basis of the vector space $\mathcal{L}(nP_{k})$, for any positive integer $n$.
\end{theorem}

\begin{remark} \label{iso between Bk + and -}
Recall our notation that 
$B_k:=B_{k,c}$ is given by the normalization of the curve with affine equation
$y^p+c y=x^{p^k+1}$.
There are $K$-isomorphisms between the curves $B_{k,c}$
for different values of $c$. For example, when $c=-1$,
we can map the curve $B_{k,-1} : y^p-y=x^{p^k+1}$ to the curve 
$B_{k,1}: Y^p+Y=X^{p^k+1}$ (where $c=+1$)
by taking the change of variable $y=aY$, $x=bX$, where $a^{p-1}=-1$ and $b^{p^k+1}=a^p$.
This is an isomorphism that maps the point at infinity to the point at infinity. 
Although both curves are defined over $\F_p$, the isomorphism is not.
The curves have a different L-polynomial over $\F_p$.
\end{remark}

Using  this isomorphism between the curves $B_{k,1}$ and $B_{k,-1}$, we have:

\begin{corollary}\label{morphismiff}
    There is a  morphism between the curves $B_{k,1}$ and $B_{\ell, 1}$, if and only if
    there is a  morphism between the curves $B_{k,-1}$ and $B_{\ell, -1}$.
\end{corollary}

\begin{proof}
All maps are defined over $K$.
    Assume a  morphism $\phi: B_{k,-1} \longrightarrow B_{\ell,-1}$ exists. 
    Let us denote by $\psi_k$ the isomorphism
    $\psi_k : B_{k,1}\longrightarrow B_{k,-1}$   that we have defined before in \Cref{iso between Bk + and -}. Then 
    $$(\psi_\ell)^{-1}\circ \phi \circ \psi_k: B_{k,1} \longrightarrow B_{\ell,1}$$ is a morphism from $B_{k,1}$ to $B_{\ell,1}$. 

    Conversely, given a  morphism $\phi: B_{k,1} \longrightarrow B_{\ell,1}$, the same argument holds.
\end{proof}

\section{Non-Galois Case }\label{maint}

In this  section, we find conditions for the existence of a (not necessarily Galois)
morphism between the curves of the family $\mathcal{B}$. 
Let us first state and prove an elementary lemma about the condition 
$p^\ell+1$ divides $p^{k}+1$.

\begin{lemma}\label{lem:divisibility cond}
    Given integers $\ell, k >1$, $p^\ell+1$ divides $p^{k}+1$ if and only if $\ell \mid k$ and $\frac{k}{\ell}$ is odd.
\end{lemma}

\begin{proof}

    $[\Leftarrow]$ Assume first $\ell\mid k$ and $\frac{k}{\ell}=m$ is odd. Then we have 
    \[
    p^k+1 = (p^{\ell})^m +1=(p^{\ell}+1)\left(\sum_{i=0}^{m-1} (-1)^i(p^{\ell})^i\right). 
    \]
    
    $[\Rightarrow]$ If  $p^\ell+1$ divides $p^{k}+1$, using the division algorithm, we have that 
    \[
    p^k+1 = (p^\ell+1)(p^{k-\ell} - p^{k-2\ell} + p^{k-3\ell} - \cdots + p^{k-m\ell}),
    \]
    for some $m\in \N$. As
    \[
    (p^{k-\ell} - p^{k-2\ell} + p^{k-3\ell} - \cdots + p^{k-m\ell}) | (p^k+1),
    \] 
    we have $p^{k-m\ell}=1$, therefore $k-m\ell = 0$. Thus, $\ell|k$ and $m=\frac{l}{k}$ is odd because otherwise we would not obtain a positive sign in the last term of the sum: 
    \[
    p^{k-\ell} - p^{k-2\ell} + p^{k-3\ell} - \cdots + p^{k-m\ell}. \qedhere
    \]
\end{proof}

We need three more lemmas which we will use in the proof.

\begin{lemma}\label{rh_ineq}
Let $p$ be an odd prime. Assume there is a 
 non-trivial morphism  $B_{k} \longrightarrow B_{\ell}$ of degree $d$.
Then 
$$ d\leq \frac{p^{k+1}-p^k-2}{ p^{\ell+1}-p^\ell-2} $$
\end{lemma}

\begin{proof}
By the Riemann-Hurwitz formula, and the fact that the genus
of $B_k$ is $(p-1)p^k/2$.
\end{proof}

\begin{lemma}\label{kl_ineq}
Let $p$ be an odd prime. Let $1<\ell <k$.
Then 
$$\frac{p^k+1}{p} > \frac{p^{k+1}-p^k-2}{ p^{\ell+1}-p^\ell-2} $$
\end{lemma}

\begin{proof}
Multiplying both sides of the claimed inequality by the positive denominator 
\(p(p^{\ell+1}-p^\ell-2)\) (this denominator is positive since \(p\ge 3\) and \(\ell\ge 2\), 
so \(p^\ell(p-1)\ge p^2\cdot 2 = 2p^2 > 2\)), we reduce to showing
\[
(p^k+1)(p^{\ell+1}-p^\ell-2) > p(p^{k+1}-p^k-2).
\]

Expanding both sides gives
\[
(p^k+1)p^\ell(p-1)-2(p^k+1) - p^{k+2}+p^{k+1}+2p>0.
\]
This can be grouped as
\[
p^k\big(p^\ell(p-1)-p^2+p-2\big) + \big(p^\ell(p-1)+2p-2\big)>0.
\]

For \(\ell\ge 2\) and \(p\ge 3\) we have
\[
p^\ell(p-1)-p^2+p-2 \;\ge\; p^2(p-1)-p^2+p-2 = p^3-2 p^2+p-2 >0,
\]
and clearly \(p^\ell(p-1)+2p-2 >0\). Hence both summands are strictly positive.
\end{proof}

\begin{lemma}\label{kl_ineq_2}
Let $p$ be an odd prime. Let $1<\ell <k$.
Then 
$$\frac{2(p^k+1)}{p^\ell+1} > \frac{p^{k+1}-p^k-2}{ p^{\ell+1}-p^\ell-2} $$
\end{lemma}

\begin{proof}
The denominators are positive for \(p\ge3\) and \(\ell\ge2\). Multiplying both sides by 
\((p^\ell+1)(p^{\ell+1}-p^\ell-2)\) we reduce to
\[
2(p^k+1)(p^{\ell+1}-p^\ell-2)>(p^{k+1}-p^k-2)(p^\ell+1).
\]
Expanding and rearranging yields
\[
p^k\big(p^{\ell+1}-p^\ell-p-3\big)+2\big(p^{\ell+1}-1\big)>0.
\]

Since \(\ell\ge2\) we have \(p^{\ell+1}-p^\ell=p^\ell(p-1)\ge p^2(p-1)\), hence
\[
p^{\ell+1}-p^\ell-p-3 \;\ge\; p^2(p-1)-p-3 \;=\; p^3-p^2-p-3.
\]
For every odd prime \(p\ge3\) this quantity is positive, and clearly 
\(p^{\ell+1}-1>0\). Therefore both summands are strictly positive, so the 
whole expression is positive, proving the claim.
\end{proof}

Here is the first main result of this paper. Part (a) is well known.

\begin{theorem}\label{thm: main}
Let $p$ be an odd prime. Let $\ell, k$ be positive integers with $1<\ell <k$. 
Let $K=\overline{\F_p}$. Let $c \in K$ be nonzero.
Let $B_k$ be the curve $y^p+c y=x^{p^k+1}$ and let 
$B_\ell$ be the curve $y^p+c y=x^{p^\ell +1}$. 
Let $P_k$ and $P_\ell$ be the points at infinity of $B_k$ and $B_\ell$,
respectively.

(a) If $p^\ell +1$  divides $p^k+1$, then
     there exists a non-trivial morphism  
    $\rho : B_{k} \longrightarrow B_{\ell}$, with the properties that
    $\rho(P_k)=P_\ell$, $\rho$ is totally ramified at $P_k$,
    and $\deg (\rho)=\frac{p^k+1}{p^\ell +1}$.

(b) Conversely, if there exists a non-trivial morphism  
    $\phi : B_{k} \longrightarrow B_{\ell}$, such that $\phi (P_k)=P_\ell$ 
    and $\phi$ is totally ramified at $P_k$, 
    and $\deg (\phi)$ is relatively prime to $p$,
    then $p^\ell +1$  divides $p^k+1$ and
    $\deg (\phi)=\frac{p^k+1}{p^\ell +1}$.
\end{theorem}

\begin{proof}
     (a)
    By \Cref{lem:divisibility cond}, if $p^\ell +1$  divides $p^k+1$
then $\ell$ divides $k$ and     $m=\frac{k}{\ell}$ is odd.
    An explicit morphism is $\rho:B_{k} \longrightarrow B_{\ell}$ defined as
$  \rho (x,y) = \left(x^{t}, y \right)$ on affine points, where 
    $t=\sum_{i=0}^{m-1} (-1)^i(p^{\ell})^i=(p^k+1)/(p^\ell +1)$,
    and mapping $P_k$ to $P_\ell$.
    The stated properties are easy to verify.

    This morphism $\rho$ could be composed with automorphisms of $B_k$ and $B_\ell$
    to yield other morphisms $ B_{k} \longrightarrow B_{\ell}$ of the same degree.
All of these morphisms will map $P_k$ to $P_\ell$ because all automorphisms of
$B_k$ and $B_\ell$ fix infinity.

    (b)
    Suppose there is a surjective  morphism $\phi: B_{k} \longrightarrow B_\ell$.
    The pullback $\phi^*$ of $\phi$ is an injective  field homomorphism $\phi^*: K(B_\ell) \longrightarrow K(B_k)$, where $ K(B_\ell)=K(x_\ell,y_\ell)$, with $y_\ell^p+c y_\ell=x_\ell^{p^\ell+1}$, 
    is the function field of $B_\ell$,
    and  $K(B_{k})=K(x_k,y_k)$ with $y_k^p+c y_k=x_k^{p^{k} +1}$ is the function field of  $B_k$.

    Let us call $u=\phi^*(x_\ell)$, $v=\phi^*(y_\ell)$ the images of the generators $x_\ell$ and $y_\ell$ of $K(B_\ell)$ under the pullback. They belong to $K(B_k)$ and, since $\phi^*$ is a homomorphism, they satisfy the following equation $v^p+c v=u^{p^\ell+1}$.

    Both $B_\ell$ and $B_{k}$ have a unique place at infinity, which we will call $P_{\ell}$ and $P_{k}$ respectively. We are going to consider now the valuation of the generators in the point at infinity for each of the curves, see \cite{Stichtenoth1973B}[Proposition 6.4.1] or \cite{HKT}[Lemma 12.1]. 
    At $P_{\ell}$, we have $v_{P_{\ell}}(x_\ell)=-p$ and $v_{P_{\ell}}(y_\ell)=-(p^\ell+1)$. Similarly, at $P_{k}$ we have that $v_{P_{k}}(x_k)=-p$ and $v_{P_{k}}(y_k)=-(p^{k}+1)$.

    Let the degree of $\phi$ be $d$. By assumption, 
    $\phi(P_{k})=P_{\ell}$ and
    $\phi$ is totally ramified at $P_{k}$.
    The valuation of $u$ and $v$ at the point $P_{k}$ is
    \[
    v_{P_{k}}(u)=v_{P_{k}}(\phi^*(x_\ell))=e_{\phi}(P_{k}) \cdot v_{\phi(P_{k})}(x_\ell)=e_{\phi}(P_{k}) \cdot v_{P_{\ell}}(x_\ell)=-dp
    \]
    \[
    v_{P_{k}}(v)=v_{P_{k}}(\phi^*(y_\ell))=e_{\phi}(P_{k}) \cdot v_{\phi(P_{k})}(y_\ell)=e_{\phi}(P_{k}) \cdot v_{P_{\ell}}(y_\ell)=-d(p^\ell+1).
    \]
    Since the Weierstrass semigroup of $B_{k}$ at $P_{k}$ is generated by $p$ and $p^{k}+1$, and $v$ has a pole of order $d(p^\ell+1)$ at $P_{k}$ (and no other poles because $\phi$ is totally ramified at the point $P_k$) 
    we have
    \begin{equation}\label{ab_eqn}
    ap+b(p^k+1)=d(p^\ell+1)
    \end{equation}
    for some nonnegative integers $a$ and $b$.
    We wish to prove that $b=1$ and $a=0$.
    This will prove that $d(p^\ell+1)=p^{k}+1$ and therefore
    that $p^\ell+1$ has to divide $p^{k}+1$, namely that $d=\frac{p^k+1}{p^\ell +1}$.

    By Theorem \ref{th: basis RR Sp} we know a basis for the Riemann-Roch space  
    $\mathcal{L}(dpP_{k})$.
Since $u$ has a pole of order $dp$ at $P_{k}$, we may express $u$ in terms of the basis
    as $u=\sum A_{ij} x_k^iy_k^j$, where the sum is over 
    all $i$ and $j$ such that $0\leq j \leq p-1$ and 
    $pi+(p^k+   1)j\leq dp$. 
      
It follows from the last inequality that we must have $j=0$ in every term, since $j>0$ would imply $d\ge \frac{p^k+1}{p}$
and this contradicts Lemmas \ref{rh_ineq} and \ref{kl_ineq} together.
Therefore $u$ is a function of $x_k$ only, and $K(u)$ is a subfield of the 
rational function field $K(x_k)$.
The only term in $u$ with the maximal pole order $dp$ is $x_k^d$, the $i=d$ term,
so  $\sum A_{ij} x_k^iy_k^j$ is equal to $g(x_k)$ where $g(x)$ is a 
monic polynomial of degree $d$.

Since $v\in \mathcal{L}(d(p^\ell +1)P_{k})$, by the same argument we 
may express $v$ in terms of the basis
    $v=\sum A_{\alpha \beta } x_k^\alpha y_k^\beta$, where the sum is over 
    all $\alpha$ and $\beta$ such that
    $0\leq \beta \leq p-1$ and $p\alpha+(p^k+   1)\beta\leq d(p^\ell +1)$.
    By construction, equality holds when $\alpha =a$ and $\beta =b$,
     and  this is the only term in this 
     sum with the maximal pole order $d(p^\ell +1)$ because each element in the basis
     has a different pole order.
We note that each term can only have $\beta=0$ or $\beta=1$, because
$\beta >1$ implies $d\ge \frac{2(p^k+1)}{(p^\ell +1)}$, which contradicts 
Lemmas \ref{rh_ineq} and \ref{kl_ineq_2}.
(In particular,  either $b=0$ or $b=1$.)
This implies that we can write $v=y_kh_1(x_k)+h(x_k)$ for some polynomials $h_1, h$.
The exponents in $h_1(x_k)$ are the values of $\alpha$ such that
$p\alpha+(p^k+   1)\leq d(p^\ell +1)$, i.e., 
corresponding to terms where $\beta =1$.
The exponents in $h(x_k)$ are the values of $\alpha$ such that
$p\alpha\leq d(p^\ell +1)$, i.e., 
corresponding to terms where $\beta =0$. 

There are two cases to consider, when $b=0$ and $b=1$. 
Note that when $b=0$ the degree of $h(x_k)$ is $a$, and when $b=1$ 
the degree of $h_1(x_k)$ is $a$. 

Assume $b$ is either 0 or 1 for now.
Substituting $u=g(x_k)$ 
and $v=y_kh_1(x_k)+h(x_k)$  into the equation
$v^p+c v=u^{p^\ell+1}$
gives
$$
y_k^p h_1(x_k)^p +h(x_k)^p+c y_k h_1(x_k) +c h(x_k)=
g(x_k)^{p^\ell +1}
$$
in the function field.
We replace $y_k^p$ by $-c y_k+x_k^{p^k+1}$. After rearranging we get
$$-c y_k \biggl(h_1(x_k)^p - h_1(x_k)\biggr)= g(x_k)^{p^\ell +1}-c h(x_k)-h(x_k)^p-x_k^{p^k+1}h_1(x_k)^p.$$
If both sides are not 0 then we obtain $y_k \in K(x_k)$ which is a contradiction. 
Therefore both sides are 0, so
$h_1(x_k)^p = h_1(x_k)$, which implies $h_1(x) \in \F_p$.
This shows that $h_1(x_k)$ is a polynomial of degree 0.

Now we assume that $b=1$. In this case we know that $\deg (h_1) =a$. 
But we just proved that $\deg (h_1) =0$.
Therefore $a=0$, and the 
proof is complete in the $b=1$ case.

It remains to consider the case $b=0$.
Assuming $b=0$, equation \eqref{ab_eqn} becomes
    $ap=d(p^\ell+1)$.
This implies that $p|d$, which contradicts the hypothesis.
\end{proof}

\section{About Galois morphisms}\label{sec: Galois morphisms}

From now on we mostly restrict ourselves to the case when $c=1$, this is $B_k: y^p+y=x^{p^k+1}$.
In this section, we consider the case in which the morphism $\phi~\colon~B_k\rightarrow~B_{\ell} $ is Galois. This is equivalent to saying that $B_{\ell}=B_k/G$, where $G$ is a subgroup of the automorphism group $\Aut(B_k)$ of $B_k$  (the equivalence follows from Lemma 11.36 of \cite{HKT}).

The automorphism group of $B_k$, $k>1$, is described in the following lemma.
It can be found in \cite{BoniniMontanucciZini}[Lemma 2.3] and \cite{HKT}[Lemma 12.9] but we write here for clarity:

\begin{theorem}[\cite{HKT} Theorems 12.1 and 12.9]\label{theo12.9}
Let $B_k$ be the curve $y^p+y=x^{p^k+1}$.
Let $\Aut(B_k)$ denote the $K$-automorphism group of $B_k$ and 
let $P_k$ be the place arising from the unique branch centered at the unique point at infinity of $B_k$. 
Then $\Aut(B_k)$ fixes the point $P_k$,  and 
$\Aut(B_k)=G_{P_k}=G^{(1)}_{P_k} \rtimes H$ with 
$p \nmid |H|$. Moreover 
    \begin{enumerate}[(i)]
        \item the cyclic group $H$ is generated by the map $\alpha$ with
        \[
        \alpha (x,y)=(\gamma x,\gamma^{p^k+1}y);
        \]
        where $\gamma$ is a $(p^k+1)(p-1)$-th root of the unity. In particular
        $|H| = (p^k+1)(p - 1)$.
        \item The elementary abelian $p$-group $G^{(1)}_{P_k}$ has order $p^{2k+1}$, and its elements are the automorphisms $\beta_{d,e}$ defined by
        \[
        \beta_{d,e}(x,y)=(x+d,y+e+d^{p^k}x-d^{p^{k+1}}x^{p}+\cdots +(-1)^{k-1}d^{p^{2k-1}}x^{p^{k-1}}).
        \]
        where $d^{p^{2k}}+(-1)^kd=0$ and $e^p+e=d^{p+1}$.
    \end{enumerate}
\end{theorem}

The Artin-Schreier curve $B_{k,-1}: y^p-y=x^{p^k+1}$ has a group of 
automorphisms known as the Artin-Schreier translations. These
are the automorphisms $(x,y)\mapsto (x,y+e)$ where $e\in \F_p$.
The quotient curve by this group is well known to be rational, the
projective line. 
With the curve $B_{k,1}: y^p+y=x^{p^k+1}$, when we refer to the Artin-Schreier translations we mean the group of automorphisms $(x,y)\mapsto (x,y+e)$ where $e^p+e=0$. These are the image of the Artin-Schreier translations for $B_{k,-1}$
under the isomorphism in \Cref{iso between Bk + and -}.
In the notation of \Cref{theo12.9} they are all the automorphisms $\beta_{0,e}$.
We shall denote this subgroup by $E$, and we will use the fact that
$B_k/E$ is rational.

We are going  to try to see if $B_\ell\cong B_k/G$.

\begin{theorem}\label{first_groups}
Let $G$ be a subgroup of $\Aut(B_k)$ of order $d$.
\begin{enumerate}
\item If $G\le H$ then $B_k/G$ cannot be isomorphic to $B_\ell$ unless
$d=\frac{p^k+1}{p^\ell +1}$.
\item If $G\le G^{(1)}_{P_k}$ then $B_k/G$ cannot be isomorphic to $B_\ell$.
\item If $G\ge G^{(1)}_{P_k}$ then $B_k/G$ cannot be isomorphic to $B_\ell$.
\end{enumerate}
\end{theorem}

\begin{proof}
 
(1)
Since $\Aut(B_k)$ fixes $P_{\infty}$, by \cite[Exercise 3.16]{StichBook}, the cover $B_k\mapsto B_k/G$ is totally ramified at $P_{\infty}$. The claim is now a consequence of \Cref{thm: main}.

(2)
Assume by contradiction that $B_{\ell}=B_k/G$ for some $G\leq G_{P_{\infty}}^{(1)}$.
Then, by \cite[Exercise 3-4]{milne2022fields}, 
\[
N_{\Aut(B_k)}(G)/G\leq \Aut(B_{\ell}).
\]
Since $ G_{P_{\infty}}^{(1)}$ is abelian of order $p^{2k+1}$, and $|G_{P_\ell}^{(1)}|=p^{2\ell+1}$, if $G$ has order $d$, it must be
\[
\frac{p^{2k+1}}{d}\leq p^{2\ell+1};
\]
which implies $d\geq p^{2(k-\ell)}$.

The Riemann-Hurwitz formula reads
\[
p^k(p-1)-2=d(p^{\ell}(p-1)-2)+\Delta.
\]
Therefore, since $d\geq p^{2(k-\ell)}$,
\[
p^k(p-1)-2\geq(p^{\ell}(p-1)-2)p^{2(k-\ell)}+\Delta.
\]
This can be read as
\[
\Delta\leq (p^{k-\ell}-1)(2(p^{k-\ell}+1)-(p-1)p^k).
\]
For any odd prime $p$ and integers $k>\ell>0$, this implies $\Delta<0$, which is not possible.

(3)
Such a subgroup $G$ contains the subgroup $E$ of Artin-Schreier translations, forcing the quotient curve to be rational.
\end{proof}

\section{Genus of the quotient group}\label{Section5}

Let $G\leq \Aut(B_k)=A\rtimes H$. 
Note that, since $\Aut(B_k)$ fixes $P_k$, $G$ does as well, and is of the form $ G=U\rtimes C $ with $ U\le A\coloneqq G^{(1)}_{P_k}$ and $ C\le H $ cyclic, following the notations of \Cref{theo12.9}, see \cite[Theorem 11.49]{HKT} or \cite[ Chapter IV, Corollary 4]{StichBook}.

We denote by $E$ the automorphism subgroup of Artin-Schreier translations, which is contained in $A$.
Since the quotient $B_k/E$ is rational, we moreover assume that $G\cap E=\{\textup{Id}_G\}$.
Let $N=(p^k+1)(p-1)$, $m=|C|$, and $r$ such that $|U|=p^r$. Recall that $|A|=p^{2k+1}$ and $|H|=N$. Moreover, let us fix $s\mid N$ such that $C=\langle\alpha^s\rangle$, and define
\[
\lambda=\gamma^s,\qquad
\mu=\gamma^{s(p^k+1)},
\]
so that $ \alpha^s(x,y)=(\lambda x,\mu y) $. The order $t$ of $ \mu $ is
\[
t:=ord(\mu)=\frac{p-1}{\gcd(s,p-1)}.
\]

In the following, we will compute the genus of $B_k/G$ as a function of $N,m,t$ and $r$.
The Riemann-Hurwitz formula reads
\[
2\mathfrak{g}(B_k)-2=|G|(2\mathfrak{g}(B_k/G)-2)+\Delta;
\]
where $\Delta$ is the different of the quotient map. The different can be computed via the Hilbert Different Formula, which tells us $\Delta=\sum_{P\in B_k(K) }d_P$ where $d_P=\sum_{i\geq 0}|G_P^{(i)}|-1$, see \cite{HKT}[Theorem 11.70].
Define $\Delta=\Delta_1+d_{P_{k}}$, with 
\[
\Delta_1=\sum_{\substack{P\in B_k(K), P\neq P_{k}}}d_P.
\]
To compute $\Delta_1$, we first describe the short orbits of $C=\langle \alpha^s\rangle$ on affine points, where $\alpha^s(x,y)=(\lambda x,\mu y)$ .
Let $S=\{P\in B_k(K) |P=(0,y)\}$, and note that $S$ consists of exactly $ p $ points:
\[
(0,0),\quad \text{and }(0,y\xi)\ \text{with }\xi^{p-1}=-1\text{ and }y^{p-1}=1.
\]

\begin{itemize}
\item If $P\notin S$, namely $P=(x,y)$ with $ x\neq 0 $, then $ \lambda^n x=x $ with $ 1\le n<m $ has no solution, so no such point lies in a short $ C $-orbit.
\item  The point $ (0,0)\in S $ is fixed by each element of $ C $, namely it forms a short $C$-orbit of size $1$.
\item If $P\in S$, and $P=(0,y\xi) $ where $y\in\mu_{p-1} $, the action of $C$ is 
$ (0,y\xi)\mapsto(0,(\mu y)\xi) $. Therefore, for any point $P\in S\setminus \{(0,0)\}$, the stabilizer $C_P$ has size $ m/t $. By the orbit-stabilizer theorem, the  $ p-1 $ points $ S\setminus \{(0,0)\}$ split into $ (p-1)/t $ $C$-orbits, each of length $ t $.
\end{itemize}

Note that for any $u\in U$ and any point $P\in S$, the image $u(P)$ of $P$ under $u$ does not belong to $S$.
Moreover, since any $p$-automorphism $\beta_{d,e}$ only fixes the point at infinity $P_{k}$, $U$ acts regularly on the affine points.

As a consequence, the short $G$-orbits on $B_k(K)\setminus \{P_{k}\}$ are precisely the $(p-1)/t$ sets $U(\mathcal{O}_1),\dots, U(\mathcal{O}_{(p-1)/t})$, where $\mathcal{O}_1,\dots, \mathcal{O}_{(p-1)/t}$ are the short $C$-orbits on $B_k(K)\setminus \{P_{k}\}$. The only special case is the point $(0,0)$ which has a short orbit in $G$ of size $p^r$.

Summing up,  by \cite[Theorem 11.57]{HKT} we have
\[
\Delta_1=\frac{p-1}{t}(mp^r-p^rt)+mp^r-p^r. 
\]

We now compute $d_{P_k}$.
Since $\Aut(B_k)$ fixes $P_{k}$, in particular $G$ fixes it, namely $G_P^{(0)}=G$. Moreover, $G_{P_{k}}^{(1)}=U$, see \cite[Theorem 11.64]{HKT}. We claim that  $G_{P_{k}}^{(i)}=\{\textup{Id}_G\}$ for any $i\geq 2$.

Since  $v_{P_{k}}(x_k)=-p$ and $v_{P_{k}}(y_k)=-(p^{k}+1)$, it follows that the element $t$ defined as $t=x^{p^{k-1}}/y$, is such that
\[
v_{P_k}(x^{p^{k-1}}/y)= p^{k-1}v_{P_k}(x)-v_{P_k}(y)=1;
\]
namely $t$ is a uniformising element at $P_{k}$.
The claim is therefore proven if we show that
\[
v_{P_k}(\beta_{d,0}(t)-t)=2\quad\text{for all }1\neq\beta_{d,0}\in U.
\]
Recall that for any $d$, $\beta_{d,0}(x,y)=(x+d,y+L_d(x))$, which implies 
\[
\beta_{d,0}(t)=\frac{x^{p^{k-1}}+d^{p^{k-1}}}{y+L_d(x)}=\frac{x^{p^{k-1}}(1+d^{p^{k-1}}/x^{p^{k-1}})}{y(1+L_d(x)/y)}=t\frac{1+\delta}{1+\epsilon};
\]
where $\delta$ and $\epsilon$ are $\delta = (d/x)^{p^{k-1}}$ and $\epsilon = L_d(x)/y$.
Note that $v_{P_{k}}(\delta)=p^k$ and $v_{P_{k}}(\epsilon)=-p^k-(-(p^k+1))=1$ because the degree of $L_d$ is $p^{k-1}$.
Therefore,
\[
\beta_{d,0}(t)-t=t\frac{\delta-\epsilon}{1+\epsilon}.
\]
Taking the valuations, we get
\[
v_{P_{k}}\left(t\frac{\delta-\epsilon}{1+\epsilon}\right)=v_{P_{k}}(t)+v_{P_{k}}\left(\frac{\delta-\epsilon}{1+\epsilon}\right)=1+v_{P_{k}}(\delta-\epsilon)-v_{P_{k}}(1+\epsilon)=1+1-0=2;
\]
where
$v_{P_{k}}(\delta-\epsilon)=\mathrm{min}(v_{P_{k}}(\delta),v_{P_{k}}(\epsilon))=1$ and 
$v_{P_{k}}(1+\epsilon)=\mathrm{min}(v_{P_{k}}(1),v_{P_{k}}(\epsilon))=0$.
So, $G_{P_{k}}^{(i)}=\{\textup{Id}_G\}$ for any $i\geq 2$.
And we obtain
$$d_{P_k}=mp^r-1+p^r-1.$$

From the  Riemann--Hurwitz formula, if we use $\Delta=\Delta_1+d_{P_{k}}$, we get that the genus of the quotient curve is
\begin{equation}\label{eq: genus quotient}
    \mathfrak{g}(B_k/G) = \frac{(p-1)(t(p^r+p^k)-p^rm)}{2p^rmt}.
\end{equation}

\section{Conditions for $B_{\ell}$ to be a quotient of $B_k$}\label{sec: Cond quot}

Since we have a formula for the genus of the quotient, 
we can derive a condition for when $B_\ell$ can be obtained as a quotient of $B_k$ by a subgroup $G$ of order $p^rm$. Starting from \Cref{eq: genus quotient} and imposing that $\mathfrak{g}(B_k/G)$ is equal to $ \mathfrak{g}(B_\ell)= p^\ell(p-1)/2$, we get the following relation
between the parameters
\begin{equation}\label{eq: B_kB_l}
    m(p^{\ell}t+1)= t(p^{k-r}+1).
\end{equation}

So, we need to study when this condition is satisfied, subject to our constraints that arise from the previous sections.
If $k$ does not divide $\ell$ then we know that there is no morphism
$B_k \longrightarrow B_\ell$.
We add the constraint that $p^\ell +1$ does not divide $p^k+1$, and we expect
no solutions. 
We will use methods from elementary number theory in
this section to show that this equation  has no solutions in certain cases.

Next we  state a well known lemma that we will use.

\begin{lemma}\label{gcd_p_plus}
$$\gcd(p^a+1,p^b+1)=
\begin{cases}
    p^{\gcd(a, b)}+1 &\text{ if } \frac{a}{\gcd(a, b)} \text{ and }  
    \frac{b}{\gcd(a, b)} \text{ are both odd}, \\
    2  &\text{ otherwise.}
\end{cases}
$$
\end{lemma}

This result allows us to separate the study of \Cref{eq: B_kB_l} in two cases, this is, when $\gcd(p^{k-r}+1, p^k+1)=2$ and when $\gcd(p^{k-r}+1, p^k+1)>2$. 

We will also use the following.

\begin{lemma}[Lemma 2.2 from \cite{Ramzy2022APT}]\label{apdp}
    Assume $A, P$ are integers with $1\leq A \leq P$. If there is an integer $D>0$ such that 
    \[
    \frac{AP+1}{DP+1} \in \Z,
    \]
    then we must have $D=A$.
\end{lemma}

From the definition of $t$ in \Cref{Section5} we know that $t|(p-1)$. Using this fact and taking $P=p^\ell$,  we have that the previous lemma gives a contradiction for $\ell \geq (k-r)/2$. Therefore, we can restrict ourselves to the case $\ell < (k-r)/2$.

We are going to use this upper bound for $\ell$ in the following result, in which we prove some conditions that $u=m/t$ satisfies.

\begin{lemma}\label{general_div}
Let $p$ be an odd prime. Assume
\begin{enumerate}
  \item $k,\ell,r,t,m\in\mathbb{N}$, 
  \item $1<\ell < k$,  $0<r<k$, and $\ell < \frac{k-r}{2}$,
  \item $\ell\mid k$ and $k/\ell$ is even,
  \item $t\mid (p-1)$,
  \item $m\mid (p^k+1)(p-1)$ and $m>1$,
  \item $m\bigl(p^\ell t+1\bigr) \;=\; t\bigl(p^{k-r}+1\bigr).$
\end{enumerate}
Let $u=m/t$. Then 
\begin{enumerate}
\item $u\in \mathbb{Z}$ and $u\ge 5$.
\item If we write $u = 2^e u_0$ with $u_0$ odd,
then $u_0$ is a common divisor of $p^{k-r}+1$ and $p^k+1$.
\item $u_0$ divides $p^r-1$.
\item $\gcd(u_0,t)=1$.
\item $u \ge p^{(k-r)/2}$
\end{enumerate}
\end{lemma}

\begin{proof}
(1) \quad Note that
$  \gcd(p^\ell t+1,t) = 1$.
Therefore, from
\[
  m(p^\ell t+1) = t(p^{k-r}+1)
\]
we deduce that $t$ divides $m$. Let $u=m/t$.
Then
\begin{equation}\label{eq:uk}
  p^{k-r}+1 = u\,\bigl(p^\ell t+1\bigr)
\end{equation}
and it is clear that $u$ divides $p^{k-r}+1$.
In particular, $u_0$ divides $p^{k-r}+1$.

Next we are going to prove by contradiction that $u\neq 1,2,3,4$.

Suppose $u=1$, then \eqref{eq:uk} yields
$  p^{k-r}+1 = p^\ell t+1$,
so $p^{k-r} = p^\ell t$ and hence $t = p^{k-r-\ell}$. Since $k>r+\ell$,
we have $t\ge p$, which contradicts $t$ dividing $p-1$.
Thus $u\neq1$.

Let us suppose now that $u=2$, then from \eqref{eq:uk} we have
$p^{k-r}+1 = 2(p^\ell t+1)$,
and reducing this modulo $p$ gives $1\equiv 2 \pmod{p}$ which is a contradiction.
 Thus $u\neq2$. 

Suppose $u=3$, then from \eqref{eq:uk} we have
$ p^{k-r} = 3p^\ell t+2$,
and reducing this modulo $p$ gives $2\equiv 0 \pmod{p}$ which is a contradiction.
This proves that $u>3$.

Suppose $u=4$, then from \eqref{eq:uk} we have
$ p^{k-r} = 4p^\ell t+3$,
and reducing this modulo $p$ gives $3\equiv 0 \pmod{p}$ which is a contradiction if $p>3$.
This proves that $u>4$ when $p>3$.
If $u=4$ and $p=3$ the equation \eqref{eq:uk} becomes
$3^{k-r}+1 = 4\,\bigl(3^\ell t+1\bigr)$
or
$3^{k-r} = 4\cdot 3^\ell \cdot t+3$.
This is a contradiction because $k-r>1$ and $\ell >1$,
so mod 9 we get $0\equiv 3 \pmod{9}$.

(2) \quad  Write $u = 2^e u_0$ with $u_0$ odd.
It remains to prove that $u_0$ is a divisor of  $p^k+1$.

Write $p-1 = ts$ with $s\in\mathbb{N}$. From $m=ut$ and the
assumption $m\mid (p^k+1)(p-1)$ we have
\[
  ut  \mid (p^k+1)(p-1) = (p^k+1)ts,
\]
hence
\begin{equation}\label{eq:u-div}
  u \mid (p^k+1)s.
\end{equation}

Let $p_0$ be the odd part of $p^k+1$, and let $s_0$ be the odd part of $s$.
Then \eqref{eq:u-div} implies $u_0 | p_0 s_0$.
It is a known fact that $\gcd(p^k+1,p-1)=2$, which implies $\gcd(p_0,s_0)=1$.

If $p_1$ is the odd part of $p^{k-r}+1$, then 
\eqref{eq:uk} implies that $u_0$ divides $p_1$.
Also, $\gcd(p_1,s_0)=1$, by the fact that $\gcd(p^{k-r}+1,p-1)=2$.

If $\gcd(u_0,s_0)>1$ then this common divisor would divide 
both $s_0$ and $p_1$, a contradiction.
Therefore $\gcd(u_0,s_0)=1$, and $u_0 | p_0 s_0$ implies $u_0$ divides $p_0$.

As $u_0|p_0$ and $u_0|p_1$, we have obtained that $u_0$ is a common divisor of $p^{k-r}+1$ and $p^k+1$.  

(3) \quad  Since $u_0$ divides $p^k+1$ and $p^{k-r}+1$,
$u_0$ divides $p^r(p^{k-r}+1)-(p^k+1)=p^r-1$.

(4) \quad Recall that $t$ divides $p-1$ and $u_0$ divides $p_0$.
Then
\[
\gcd(p^k+1,p-1)=2 
\implies \gcd(p_0,p-1)=1 
\implies \gcd(u_0,t)=1.
\]

(5) \quad Since $t<p$ we get $tp^\ell < p^{\ell +1} \le p^{(k-r)/2}$
using the assumption $\ell < \frac{k-r}{2}$.

From the equation $p^{k-r}+1 = u\,\bigl(p^\ell t+1\bigr)$
we then get $p^{k-r}+1 \le u \cdot p^{(k-r)/2}$.
It follows that $u \ge p^{(k-r)/2}$.
\end{proof}

We are going to use the previous lemma in order to first prove the case when $\gcd(p^{k-r}+1,p^k+1)=2$. We will use later this result in the general case.

\begin{lemma}\label{general_div2}
Let $p$ be an odd prime. Assume
\begin{enumerate}
  \item $k,\ell,r,t,m\in\mathbb{N}$, 
  \item $1<\ell < k$,  $0<r<k$, and $\ell < \frac{k-r}{2}$, 
  \item $\ell\mid k$ and $k/\ell$ is even,
  \item $t\mid (p-1)$,
  \item $m\mid (p^k+1)(p-1)$ and $m>1$,
  \item $\gcd(p^{k-r}+1,p^k+1)=2$.
\end{enumerate}
 Then there are no solutions
to the equation
$m\bigl(p^\ell t+1\bigr) \;=\; t\bigl(p^{k-r}+1\bigr).$
\end{lemma}

\begin{proof}
Suppose there is a solution.
Let $u=m/t$ and write $u=2^eu_0$ where $u_0$ is odd.

Since $k/\ell$ is even, then $k$ must be even.
Since  $k$ is even, we have 
$p^k\equiv 1 \pmod4$.
This implies that $p^k+1\equiv 2 \pmod4$, so
$p^k+1$ is divisible by 2, and not by 4.

Suppose  that 
$\gcd(p^{k-r}+1,p^k+1)=2$.
Then $u_0=1$ by Lemma \ref{general_div}.
So $u=2^e$ for some $e\ge 2$ (because $u>4$).
Then
$p^{k-r}+1 = 2^e\,\bigl(p^\ell t+1\bigr)$
which implies $p^{k-r}\equiv 3 \pmod4$.
This implies that $p\equiv 3 \pmod4$ and $k-r$ is odd.

Recall that  $m=ut=2^et$ for some $e\ge 2$, so $4|m$.
If $p>3$ then $e>2$ by the previous lemma,
so in fact $8|m$. 
Since $m|(p-1)(p^k+1)$ we get $8|(p-1)(p^k+1)$.
Note that 2 divides $p-1$, and 4 does not, since $p\equiv 3 \pmod4$.
Therefore 4 must divide $p^k+1$, but $p^k+1  \equiv 2 \mod 4$, because $k$ is even. So we  have arrived to a contradiction.

If $p=3$, and $e>2$, then $8|m$ and the same argument holds.
If $e=2$ then we do not necessarily have $8|m$.
The equation $p^{k-r}+1 = 2^e\,\bigl(p^\ell t+1\bigr)$ becomes
$3^{k-r}+1 = 4\,\bigl(3^\ell t+1\bigr)$
or
$3^{k-r} = 4\cdot 3^\ell \cdot t+3$.
This is a contradiction because $k-r>1$ and $\ell >1$,
so mod 9 we get $0\equiv 3 \pmod{9}$.
\end{proof}

In the following lemma we are going to add an upper bound for the difference $k-r$ as a hypothesis. The bound $k-r\leq \ell^2$ appears naturally from the proof.

\begin{lemma}\label{general_div3}
Let $p$ be an odd prime. Assume
\begin{enumerate}
  \item $k,\ell,r,t,m\in\mathbb{N}$, 
  \item $1<\ell < k$,  $0<r<k$,  and
$2\ell < k-r\le \ell^2$.
  \item $\ell\mid k$ and $k/\ell$ is even,
  \item $t\mid (p-1)$,
  \item $m\mid (p^k+1)(p-1)$ and $m>1$.
\end{enumerate}
 Then there are no solutions
to the equation
$m\bigl(p^\ell t+1\bigr) \;=\; t\bigl(p^{k-r}+1\bigr)$.
\end{lemma}

\begin{proof}
Suppose there is a solution.
Let $u=m/t$ and let $M=tp^\ell +1$. 
Write the equation as
$p^{k-r}+1=uM$,
which implies \begin{equation}\label{cong1}
p^{k-r}\equiv -1 \pmod{M}.
\end{equation}

We start with the congruence 
\begin{equation}\label{cong2}
tp^\ell \equiv -1 \pmod{M}.
\end{equation}
Multiply this by $p^{k-r-\ell}$ to obtain
\[
tp^{k-r}\equiv -p^{k-r-\ell} \pmod{M}.
\]
Using \eqref{cong1} this becomes
\[
t\equiv p^{k-r-\ell} \pmod{M}.
\]
We begin a Euclidean algorithm type procedure.
Multiply by $t$ to get
\[
t^2\equiv tp^{k-r-\ell} \pmod{M}.
\]
Write $p^{k-r-\ell}$ as $p^\ell \cdot p^{k-r-2\ell}$, to get
\[
t^2\equiv tp^\ell \cdot p^{k-r-2\ell} \pmod{M}.
\]
Using \eqref{cong2} we get
\[
t^2\equiv - p^{k-r-2\ell} \pmod{M}.
\]
This completes the Euclidean algorithm type step.
Repeating this step (multiply by $t$, and replace $tp^\ell$ with $-1$) we get
\[
t^3\equiv  p^{k-r-3\ell} \pmod{M}.
\]
After doing this Euclidean algorithm step $i-1$ times we get
\[
t^i\equiv (-1)^{i-1} p^{k-r-i\ell} \pmod{M}.
\]
The $\pm$ sign actually will not matter, as we will see.

Suppose that $k-r=Q \ell +R$ by division of $k-r$ by $\ell$, where $0\le R < \ell$. 
Then  after $Q-1$ steps  we will have
\[
t^Q\equiv (-1)^{Q-1} p^{k-r-Q\ell} \pmod{M}
\]
or
\[
t^Q\equiv (-1)^{Q-1} p^{R} \pmod{M}.
\]
We claim that this congruence is an equality of integers.
To prove this, 
\begin{align*}
|t^Q- (-1)^{Q-1} p^{R}| &\le t^Q+p^R \quad \text{triangle inequality}\\
&<p^Q+p^R \quad \text{since $t<p$}\\
&\le p^\ell+p^R \quad \text{$Q\le \ell$ since $k-r\le \ell^2$}\\
&<p^\ell+p^\ell \quad \text{since $R<\ell$}\\
&=2p^\ell\\
&\le tp^\ell \quad \text{assuming $t>1$}\\
&<tp^\ell +1=M.
\end{align*}
Note that this inequality also holds when $t=1$, and the proof is easier:
\[
|t^Q- (-1)^{Q-1} p^{R}| =|1\pm p^R|\le 1+p^R<1+p^\ell=M.
\]
Therefore, for any $t$,  $t^Q- (-1)^{Q-1} p^{R}$ is an integer which is divisible by $M$
and whose absolute value is $<M$.
The only such integer is 0. 
So we have an equality of integers
\[
t^Q= (-1)^{Q-1} p^{R}.
\]
Since $p$ does not divide $t$, this is only possible when $t=1$, 
$R=0$ and $Q$ is odd, when both sides are 1. 
In this case, moreover, $\ell$ divides $k-r$ and the quotient 
$(k-r)/\ell$ is odd. 

Recall that  $\ell |k$ and $k/\ell $ is even.
If we also have $\ell$ divides $k-r$ and 
$(k-r)/\ell$ is odd,
together these imply $\ell |r$ and $r/\ell$ is odd.
This implies $\nu_2(r)< \nu_2(k)$.
Under the assumption $\gcd(p^{k-r}+1,p^k+1)=p^g+1>2$,
we showed that $\nu_2(r)>\nu_2(k)$,
so we would have a contradiction.

Observe that we still have one unsolved case, this is when $t=1$, $\gcd(p^{k-r}+1, p^k+1)=2$ and $\ell$ divides $k-r$ and $(k-r)/\ell$ is odd. However, this case is solved by \Cref{general_div2}.
\end{proof}

\section{Results} \label{sec:inf family}

In this last section, 
we assume that $B_k=B_{k,-1}$, so $B_k$ denotes the 
Artin-Schreier curve $y^p-y=x^{p^k+1}$.
We will give two results. In the first one we describe an infinite subfamily of curves $B_k$ for which there is no Galois morphism between them, the ones of the form $k=2^a$. The second result tells us that there is no Galois map between the curves $B_k$ and $B_\ell$ when $k<\ell^2$, except in the obvious case $p^\ell+1\mid p^k+1$, in which a map does exist.

\begin{corollary}\label{result1}
If $k=2^a$, there is no Galois map from a curve of type $B_k$ to one of type $B_{\ell}$. 
In particular, there is no Galois morphism  $B_{2^a}\longrightarrow B_{2^b}$, even though as soon as $b<a$, the L-polynomial of $B_{2^b}$ divides the one of $B_{2^a}$.
\end{corollary}

\begin{proof} 
Assume $k=2^a$ and that there is a Galois morphism 
$B_{k}\longrightarrow B_\ell$, with Galois group $G$ where
$|G|=mp^r$. 
By  \Cref{morphismiff} there is a Galois morphism 
$B_{k,1}\longrightarrow B_{\ell,1}$.
As explained at the start of \cref{sec: Cond quot},
\eqref{eq: B_kB_l} has a solution.
Assuming $r>0$, if $k$ is a power of 2 then $k/\gcd(k,r)$ must be even, 
no matter what $r$ is.
Lemma \ref{general_div2}  gives a contradiction.
If $r=0$ then case 1 of Theorem \ref{first_groups} applies, 
to obtain a contradiction.

When $\ell =2^b$ and $b<a$ it is shown in \cite{MCGUIRE20193341}
that the L-polynomial of $B_\ell$ divides that of $B_k$.
\end{proof}

The next result uses a different Lemma from \cref{sec: Cond quot}.

\begin{corollary}\label{result2}
If $k<\ell^2$, there is no Galois map from a curve of type $B_k$ to one of type $B_{\ell}$, except when $|G|=(p^k+1)/(p^\ell +1)$.
\end{corollary}

\begin{proof} Suppose there is such a Galois map, with Galois group $G$ where
$|G|=mp^r$. 
By  \Cref{morphismiff} there is a Galois morphism 
$B_{k,1}\longrightarrow B_{\ell,1}$.
As explained at the start of \cref{sec: Cond quot},
\eqref{eq: B_kB_l} has a solution.
Assuming $r>0$, 
Lemma \ref{general_div3}  gives a contradiction when $\ell <(k-r)/2$,
and  Lemma \ref{apdp} proves the case $\ell \ge (k-r)/2$.
If $r=0$ then case 1 of Theorem \ref{first_groups} applies, 
to obtain the result.
\end{proof}

Another way of stating Corollary \ref{result2} is that there is no Galois cover 
$B_{k}\longrightarrow B_\ell$ for any $\ell > \sqrt{k}$
(apart from the exception).
For example, there is no Galois cover 
$B_{100}\longrightarrow B_\ell$ for any $\ell > 10$,
apart from when $\ell$ divides 100 and $100/\ell$ is odd (e.g. $\ell=20$).

As a special case of Corollary \ref{result2}, we have the following.

 \begin{corollary}
 There is no Galois cover
$B_{2\ell}\longrightarrow B_\ell$ for any integer $\ell > 2$,
even though the L-polynomial of $B_\ell$ divides that of $B_{2\ell}$.
\end{corollary}

Furthermore, given $\ell>2$, there is no Galois cover
$B_{2\ell}\longrightarrow B_n$ for any integer $n\ge \ell $.

\section*{Acknowledgments}

This publication has emanated from research supported in part by a grant from Taighde Éireann – Research Ireland under Grant number 18/CRT/6049. For the purpose of Open Access, the author has applied a CC BY public copyright licence to any Author Accepted Manuscript version arising from this submission.

\bibliography{biblio}
	\bibliographystyle{alpha}

\end{document}